\documentclass[11pt]{article}
\usepackage{latexsym,amssymb,amsmath,amsfonts,amsthm}
\usepackage{graphics,graphicx,mathrsfs,subfigure}
\usepackage{color,xspace}
\newcommand{\R}{{\mat R}}

\newcommand{\N}{{\mat N}}
\newcommand{\C}{{\mat C}}
\newcommand{\Sp}{{\mat S}}

\newcommand{\be}{\begin{eqnarray}}
\newcommand{\ben}{\begin{eqnarray*}}
\newcommand{\en}{\end{eqnarray}}
\newcommand{\enn}{\end{eqnarray*}}

\newcommand{\mat}{\mathbb}

\newtheorem{theorem}{Theorem}[section]
\newtheorem{lemma}[theorem]{Lemma}

\newtheorem{remark}[theorem]{Remark}

\definecolor{rot}{rgb}{1,0,0}
\definecolor{hw}{rgb}{0,0,1}

\begin{document}
\renewcommand{\theequation}{\arabic{section}.\arabic{equation}}
\title{\bf
 Increasing stability for the inverse source problems in electrodynamics
}
\author{
Suliang Si\thanks{School of Mathematics and Statistics, Shandong University of Technology,
Shandong, 255049, China ({\tt sisuliang@amss.ac.cn})}}
\date{}



\maketitle

\begin{abstract}
We are concerned with increasing stability in the inverse source problems for the time-dependent Maxwell equations in $\R^3 $, where the source term is compactly supported in both time and spatial variables. By using the Fourier transform, sharp bounds of the analytic continuation and the Huygens' principle, increasing stability estimates of the $L^2$-norm of the
source function are obtained. The main goal of this paper is to understand increasing stability for the Maxwell equations in the time domain.
\end{abstract}

%

\section{Introduction}
\subsection{statement of the problem}
Consider the time-dependent Maxwell equations in a homogeneous medium 
\begin{equation}\label{EH}
\mu\partial_t\textbf{H}(\textbf{x},t)+\nabla\times \textbf{E}(\textbf{x},t)=0, \quad \varepsilon\partial_t\textbf{E}(\textbf{x},t)-\nabla\times \textbf{H}(\textbf{x},t)=-\sigma \textbf{E}+\tilde{\textbf{F}}(\textbf{x},t), \   \textbf{x}\in\R^3, \ t>0,
\end{equation}
where $\textbf{E}$ and $\textbf{H}$ are the electric and magnetic fields, respectively, the source function $\tilde{\textbf{F}}$ is known as the electric current density, $\varepsilon$ and $\mu$ are the dielectric permittivity and the magnetic permeability, respectively, and $\sigma$ is the electric conductivity and is assumed to be zero. 

Let $n=\frac{\varepsilon}{\mu}$. 
Eliminating the magnetic field $\textbf{H}$ from (\ref{EH}), we obtain the Maxwell system for the electric field $\textbf{E}$:
\begin{equation}\label{eq1}
n\partial_t^2\textbf{E}(\textbf{x},t)+\nabla\times(\nabla\times \textbf{E}(\textbf{x},t))=\frac{1}{\mu}\partial_t\tilde{\textbf{F}}(\textbf{x},t)=:\textbf{F}(\textbf{x},t), \  \textbf{x}\in\R^3, t>0,
\end{equation}
which is supplemented by the homogeneous initial conditions
\begin{equation}\label{eq2}
\textbf{E}(\textbf{x},0)=\partial_t\textbf{E}(\textbf{x},0)=0, \  \textbf{x}\in\R^3.
\end{equation}
In this work, we consider the following inverse source problems:

\noindent

\textit{\textbf{IP1}}:
Assume $\textbf{F}(\textbf{x},t)=\textbf{f}(\textbf{x})g(t)$, where $\textbf{f}\in\R^3,\ g\in\R$ have compact supports such that supp $\textbf{f}$ $\subset\{\textbf{x}\in\mathbb{R}^3| \ |\textbf{x}|<R\}$ and supp $g$ $\subset(0,T_0)$ where $R>0,\ T_0>0$ are constants. Then we concern the inverse problem of recovering the source term $\textbf{f}(\textbf{x})$ from the Dirichlet boundary data measured on $\partial B_R\times(0,T_0)$.

\textit{\textbf{IP2}}:
Assume  $n\in(0,b^2)$ for $b>1$. The source function $\textbf{F}(\textbf{x},t)$ has compact supports. The inverse problem is to determine  source term $\textbf{F}(\textbf{x},t)$ from the measurement $\textbf{E}(\textbf{x},t)$, $\textbf{x}\in\partial B_R$, $t\in(0,T_0)$.

\textit{\textbf{IP3}}:
Assume $\textbf{F}(\textbf{x},t)=\textbf{f}(x_1,x_2,t)g(x_3)$, where the function $\textbf{f}\in\R^3$ and $g\in\R$ have compact supports. The inverse problem is to determine source term $\textbf{f}$ from the boundary observation data $\textbf{E}(\textbf{x},t)$, $\textbf{x}\in\partial B_R$, $t\in(0,T_0)$.
\subsection{Motivations}
Motivated by significant applications, the inverse source problems, as an important research subject
in inverse scattering theory, have continuously attracted much attention by many researchers \cite{Arridge1999Optical,Bleistein1977Nonuniqueness,
Isakov1990inverse,Jiang2016Inverse,Yamamoto1999Stability}. Consequently, a great deal of mathematical and numerical results are available, especially for
the acoustic waves or the Helmholtz equations. In general, it is known that there is no uniqueness for
the inverse source problem at a fixed frequency due to the existence of non-radiating sources \cite{Hauer2005On}.
Therefore, additional information is required for the source in order to obtain a unique solution, such as
to seek the minimum energy solution \cite{MD1999}. From the computational point of view, a more challenging issue
is the lack of stability. A small variation of the data might lead to a huge error in the reconstruction.
Recently, it has been realized that the use of multi-frequency data is an effective approach to overcome the
difficulties of non-uniqueness and instability which are encountered at a single frequency. The first increasing
stability result in \cite{Bao2010multi-frequency} was obtained in a more particular case by quite different (spatial
Fourier analysis) methods. A topical review can be
found in \cite{BCLZ2014} on the inverse source problems as well as other inverse scattering problems by using multiple
frequencies to overcome the ill-posedness and gain increased stability. 

For electromagnetic waves, Ammari et al. \cite{ABF2002} showed uniqueness and stability, and presented an inversion
scheme to reconstruct dipole sources based on a low-frequency asymptotic analysis of the time-harmonic
Maxwell equations. In \cite{AM2006}, Albanese and Monk discussed uniqueness and non-uniqueness of the inverse
source problems for Maxwell’s equations. A monograph can be found in \cite{RK1994} on general inverse problems for
Maxwell’s equations.  In \cite{BaoStability2020}, Bao et al. develop new techniques and establish an
increasing stability theory in the inverse source scattering problems for electromagnetic
waves. Due to the inverse source problem of electromagnetic waves, there is a phenomenon of increasing stability. A natural question is whether there is also some increasing stability in the inverse source problem of time-dependent Maxwell equations.
Thus the main goal of this paper is to understand increasing stability for the Maxwell equations in the time domain.

\subsection{Known results}
Inverse sourse problems have many significant applications in scientific and engineering areas. For instance, detection of submarines and non-destructive measurement of industrial objects can be regarded as recovery of acoustic sources from boundary measurements of the pressure. Other application include biomedical imaging optical tomography \cite{Arridge1999Optical,Isakov1990inverse}, and geophysics. For a mathematical overview of various inverse source problems, we can see that uniqueness and stability are discussed in \cite{Isakov1990inverse}. For inverse source problems in time domain, it is solved as hyperbolic systems by using Carleman estimate \cite{Klibanov1999Inverse} and unique continuation; we refer to \cite{Choulli2006Some, Jiang2016Inverse,Yamamoto1999Stability,Yamamoto1999Uniqueness} for an incomplete list. For time-harmonic inverse source problems, it is well-known that there is no uniqueness for the inverse source problem with a single frequency due to the existence of non-radiating sources \cite{Bleistein1977Nonuniqueness,Hauer2005On}. Therefore, the use of multiple frequencies data is an effective way to overcome non-uniqueness and has received a lot of attention in recent years. \cite{Eller2009Acoustic} show the uniqueness and numerical results for Helmholtz equation with multi-frequency data. And in \cite{Bao2010multi-frequency}, Bao et al. firstly get increasing stability for Helmholtz equation by direct spatial Fourier analysis methods. Then in \cite{Cheng2016Increasing}, a different method involving a temporal
Fourier transform, sharp bounds of the analytic continuation to higher wave numbers were used to derive increasing stability bounds for the three dimensional Helmholtz equation. Also, \cite{Cheng2016Increasing} firstly combined the Helmholtz equation and associated hyperbolic equations to get the stability results. Later in \cite{Li2017Increasing} and \cite{BaoStability2020}, increasing stability were extended to Helmholtz equation and Maxwell's equation in three dimension.

Motivated by those works, we are interested in the increasing stability for the time-dependent Maxwell equations. Since by Fourier transform, inverse source problem of the time-dependent Maxwell equations can be reduced to that of the associated Electromagnetic with multi-frequency data, we derive our increasing stability estimate by using sharp bounds of analytic continuation given in \cite{Cheng2016Increasing}. Instead of Cauchy data on measure boundary used in the works mentioned above, we use only Dirichlet data on the lateral boundary. For the uniqueness, we just prove it in time domain without Carleman estimate.

The rest of this paper is organized as follows. In Section \ref{Thm} and \ref{pre}, we state our main results and well-posedness of the direct problem. Sections \ref{pre} is devoted to the increasing stability of inverse problem 1. In Section \ref{th3}, we prove general source terms in homogeneous medium. Section \ref{th4}, we establish the increasing stability of general source term.
Finally, the increasing stability of time-dependent source term can be established by using the boundary Dirichlet data.

\section{Main results}\label{Thm}

Let $B_R=B(0, R)=\{\textbf{x}\in \R^3 | \ |\textbf{x}|<R\}$ for  $R>0$. 
We always assume that the source function $\textbf{F}(\textbf{x},t)$ is required to be real-valued, which implies that 
$\widehat{\textbf{F}}(-\xi,-\omega)=\overline{\textbf{F}}(\xi,\omega)$ for all 
$(\xi,\omega)\in \R^3\times\R$. In addition, we  also assume that $\nabla\cdot \textbf{\textbf{F}}(\textbf{x},t)=0$. It follows from that the source term $\textbf{F}(\textbf{x},t)$ can be decomposed into a sum of radiation and non-radiating parts. The non-radiating part cannot be determined and gives rise to the non-uniqueness issue. By the divergence-free condition of $\textbf{F}(\textbf{x},t)$, we eliminate non-radiating sources in order to ensure the uniqueness of the inverse problem.
First we show increasing stability result for the time-dependent inverse problem (IP1). 
Let $\textbf{F}(\textbf{x},t)=\textbf{f}(\textbf{x})g(t)$,  where $\textbf{f}\in\R^3$ is compactly supported in $B_R$ and $g\in\R$ is supported in $(0,T_0)$ for some $T_0>0$.
It is supposed that $\textbf{f}(\textbf{x})\in H^1(\R^3)^3$ and $g(t)\in H^4(0,\infty)$.
The one-dimensional Fourier transform of $g$ with respect to the time variable $t$ is defined as
\[\widehat{g}(\omega):=\frac{1}{\sqrt{2\pi}}\int_\R g(t)e^{-i\omega t}dt=
\frac{1}{\sqrt{2\pi}}\int_0^\infty g(t)e^{-i\omega t}dt.
 \]
We suppose there exist  a number $b>1$ and a constant $\delta>0$ such that 
\be\label{af}
|\widehat{g}(\omega)|\geq \delta>0 \quad\mbox{for all}\quad \omega\in (0,b).
\en

Physically, the parameter $b$ in (2.4) is associated with the bandwidth of the temporal signal $g(t)$. The condition (2.4) covers a large class of functions. For example, if 
$g(t)=e^{-(t-1)^2/\eta}\chi(t)$ with some $\eta>0$ and $\chi(t)\in C_0^\infty(\R)$ such that $\chi(t)=0$ for  $t\notin[0, T_0]$, the one can always find the parameters $b>0$ and $\delta>0$ such that (2.4) holds true. 
Since the source function is real valued, (2.4) holds true for all $\omega\in(-b, b)$.
We remark that the interval $(0, b)$ in (2.4) can be replaced by $(\omega_0-b, \omega_0+b)$ for some $\omega_0\geq 0$. In this paper we take $\omega_0=0$ for simplicity.  The condition (2.8) is also similar.

Define a boundary operator 
\begin{equation}
T(\textbf{E}\times\nu)=(\nabla\times\textbf{E})\times\nu \quad \mbox{on}   \   \partial B_R\times(0,T_0),
\end{equation}
where $\nu$ is the unit normal vector on $\partial B_R$.

In the following theorem, we establish increasing stability estimate of the $L^2$-norm of $\textbf{f}$ about $b$.  
\begin{theorem}\label{TH1} 
Let the condition \eqref{af} hold and
let $T>2R+T_0$.  Assume that $g(t)$ is given and  $||\textbf{f}||_{H^1(\R^3)^3}\leq M$ where $M>1$ is a constant. If $\nabla\cdot \textbf{f}(x)=0$, then there exists a constant $C>0$ depending on $\delta$, $T_0$, $n$ and $R$ such that 
\begin{equation}\label{es0}
   ||\textbf{f}||_{L^2(\R^3)^3}^2\leq  C(b^5\epsilon^2+\frac{M^2}{b^{\frac{4}{3}}|\ln \epsilon|^{\frac{1}{2}}}),
\end{equation}
where $\epsilon=\Big(\int_0^T\int_{\partial B_R}\big(|T(\textbf{E}\times\nu)|^2+|\textbf{E}\times\nu|^2 \big)ds(\textbf{x})dt\Big)^{\frac{1}{2}}$.
\end{theorem}
\begin{remark}
There are two parts in the stability estimates: the first  parts are the data discrepancy, while this second parts comes from the high frequency tail of the function. 
The coefficient appearing in the Lipschitz part of $(\ref{es0})$ is polynomial type. However,  since $b$ is fixed in practice,  these coefficients are constants and do not pose any problem. 
It is clear to conclude that the ill-posedness of the inverse time-dependent source problem decreases as the parameter $b$ increases. This is in consistent with the increasing stability results of \cite{BaoStability2020} in the frequency domain: the ill-posedness decreases when the width of the wavenumber interval $(0, b)$ increases. 
\end{remark}
Next, we state the stability estimate for the second inverse problem (IP2).
Let $n\in(0,b^2)$, $b>1$. In the following theorem, we establish the increasing stability estimate of $L^2$-norm of $\textbf{F}$ about $b$.
Let $\textbf{E}=\textbf{E}(\textbf{x},t;n)$ satisfy (\ref{eq1}) and (\ref{eq2}).
\begin{theorem}\label{TH3} 
Let $\textbf{F}(\textbf{x},t)\in H^4([0,\infty); H^1(\R^3))^3$ be such that supp $\textbf{F}(\textbf{x},t)$ $\subset B_R\times(0,T_0)$ and let $T>2Rn+T_0$. Assume that there exists $M>1$ and $\|\textbf{F}\|_{H^4([0,\infty); H^1(\R^3))^3}\leq M$.
Then there exists a constant $C>0$ depending on  $T_0$ and $R$ such that
\begin{align}\label{es2}
\|\textbf{F}\|^2_{L^2(\R^4)} \leq  C\Big(b^{11}\,\epsilon^2+\frac{M^2}{b|\ln\epsilon|^{\frac{2}{5}}}\Big).
\end{align}
where $\epsilon=\sup\limits_{n\in(0,b^2)}\Big(\int_0^T\int_{\partial B_R}\big(|T(\textbf{E}\times\nu)|^2+|\textbf{E}\times\nu|^2 \big)ds(\textbf{x})dt\Big)^{\frac{1}{2}}$.
\end{theorem}
Finally, we present the increasing stability for the third inverse problem (IP3).
Assume that $\textbf{F}(\textbf{x},t)$ takes the form
\begin{equation}\label{3F}
\textbf{F}(\textbf{x},t)=\textbf{f}(\tilde{x},t)g(x_3), \quad \tilde{x}=(x_1,x_2)\in\R^2,  \quad   t\in(0,+\infty),
\end{equation}
where $\textbf{f}(\tilde{x},t)$ has compact supports such that supp $\textbf{f}$ $\subset$ $\widetilde{B}_{R_0}\times(0,T_0):=\{(\tilde{x},t)\in\R^2\times\R
|\ |\tilde{x}|<R_0,\ 0<t<T_0\}$. Suppose that $g$ is given and supported in $(-R_0, R_0)$ for some $0<R_0<R/\sqrt{2}$. 

The Fourier transform of  $g(x_3)$ of space variable $x_3$ is given by
\begin{equation}
\widehat{g}(\xi_3)=\frac{1}{\sqrt{2\pi}}\int_{\R^3}g(x_3)e^{-i\xi_3\cdot x_3}dx_3.
\end{equation}
We suppose
\be\label{g}
|\widehat{g}(\xi_3)|\geq \delta>0 \quad\mbox{for all}\quad \xi_3\in (-b,b),
\en
where $b>1$.
\begin{theorem}\label{TH4}
Let $\textbf{f}(\tilde{x},t)\in H^4([0,\infty); H^1(\R^2))^3$ be such that supp$\textbf{f}(\tilde{x},t)\subset \widetilde{B}_{R_0}\times(0,T_0)$ and let $T>2R+T_0$.
Assume that there exists $M>1$ such that $\|\textbf{f}\|_{H^4([0,\infty); H^1(\R^2))^3}\leq M$.
 Then there exist constants $\alpha\in(0,1)$ and $C>0$ depending on $n$, $\delta$, $M$, $T_0$ and $R$  such that
\begin{align}\label{es3}
\|\textbf{f}\|^2_{L^2(\R^3)^3} \leq C\Big( b^7\,\epsilon^2+b^5\,\,e^{2b(1-\alpha)}\epsilon^{2\alpha}+\frac{M^2}{b^{\frac{4}{3}}|\alpha\ln\epsilon|^{\frac{1}{2}}} \Big).
\end{align}
where $\epsilon=\Big(\int_0^T\int_{\partial B_R}\big(  |T(\textbf{E}\times\nu)|^2+|\textbf{E}\times\nu|^2 \big)ds(\textbf{x})dt\Big)^{\frac{1}{2}}$.
\end{theorem}

\section{Preliminaries and Proof of Theorem \ref{TH1}}\label{pre}
In this section. The electric current density is assumed to take the form
\[\textbf{F}(\textbf{x},t)=\textbf{f}(\textbf{x})g(t)\]
where $g$ is given.
Let $\textbf{E}^{inc}$ and $\textbf{H}^{inc}$ be the electric and magnetic plane waves. Explicitly, we have 
\begin{equation}\nonumber
\textbf{E}^{inc}=\textbf{p}e^{-i(\kappa_p \textbf{x}\cdot \textbf{d}+\omega t)}, \quad  \textbf{H}^{inc}=\textbf{q}e^{-i(\kappa_p \textbf{x}\cdot\textbf{d}+\omega t)}
\end{equation}
where $\textbf{d}\in\Sp^2$ is a unit vector and $\textbf{p}$, $\textbf{q}$ are two unit polarization vectors satisfying $\textbf{p}\cdot\textbf{d}=0$, $\textbf{q}=\textbf{d}\times\textbf{p}$.

If $\kappa_p^2-n\omega^2=0$, it is easy to verify that $\textbf{E}^{inc}$ and $\textbf{H}^{inc}$ satisfy the homogeneous Maxwell equations in $\R^3$:
\begin{equation}\nonumber
n\partial_t^2\textbf{E}(\textbf{x},t)+\nabla\times(\nabla\times \textbf{E}(\textbf{x},t))=0 
\end{equation}
and 
\begin{equation}\nonumber
n\partial_t^2\textbf{H}(\textbf{x},t)+\nabla\times(\nabla\times \textbf{H}(\textbf{x},t))=0.
\end{equation}
Multiplying the both sides of (\ref{eq1}) by $\textbf{E}^{inc}$, we obtain
\begin{equation}\label{eq2.12}
  \int_0^T\int_{B_R} (n\partial_t^2\textbf{E}+\nabla\times(\nabla\times\textbf{E})\textbf{E}^{inc}\ \mathrm{d}\textbf{x}\mathrm{d}t=\int_0^T\int_{B_R} \textbf{f}(\textbf{x}) g(t) \textbf{E}^{inc}(\textbf{x},t)\ \mathrm{d}\textbf{x}\mathrm{d}t.
\end{equation}
Integrating by parts, one deduces from the left hand side of (\ref{eq2.12}) that
\begin{align*}
  \int_0^T\int_{B_R}&(n\partial_t^2\textbf{E}+\nabla\times(\nabla\times\textbf{E})\textbf{E}^{inc}\ \mathrm{d}\textbf{x}\mathrm{d}t \\
  =&\int_{B_R} n(\partial_{t}\textbf{E}(\textbf{x},t) \textbf{E}^{inc}(\textbf{x},t)- \textbf{E}(\textbf{x},t) \partial_{t}\textbf{E}^{inc}(\textbf{x},t))\big|_0^T\;\mathrm{d}\textbf{x}\\
  &\quad+\int_0^T\int_{\partial B_R} \big(\nu\times(\nabla\times \textbf{E})\cdot\textbf{E}^{inc}-\nu\times(\nabla\times \textbf{E}^{inc})\cdot \textbf{E}\big) ds(\textbf{x})dt.\\
=&-\int_0^T\int_{\partial B_R} \big((T(\textbf{E}\times\nu)\cdot \textbf{E}^{inc})+(\textbf{E}\times\nu)\cdot(\nabla\times \textbf{E}^{inc})\big)ds(\textbf{x})dt.
\end{align*}
Note that, in the last step we have used the fact that $\textbf{E}(\textbf{x},t)=0$ when $|\textbf{x}|<R$ and $t>T_0+2R$, which follows straightforwardly from Huygens' principle (see e.g.,
 \cite[Lemma 2.1]{BaoHu}). This implies $\textbf{E}(\textbf{x},T)=\partial_t \textbf{E}(\textbf{x},T)=0$ for $\textbf{x}\in B_R$ and $T>T_0+2R$. Hence, the integral over $B_R$ on the left hand side of the previous identity vanishes. 
\begin{equation}\label{pf}
\int_{B_R}\textbf{p}e^{-i\kappa_p\textbf{x}\cdot\textbf{d}}\textbf{f}(\textbf{x})d\textbf{x}\int_0^Tg(t)e^{-i\omega t}dt=
-\int_0^T\int_{\partial B_R}\big((T(\textbf{E}\times\nu)\cdot \textbf{E}^{inc})+(\textbf{E}\times\nu)\cdot(\nabla\times \textbf{E}^{inc})\big)ds(x)dt
\end{equation}
By the definition of $\textbf{E}^{inc}$, one can check that
\begin{equation}\nonumber
\nabla\times\textbf{E}^{inc}=-i\kappa_p\textbf{d}\times \textbf{p}e^{-i(\kappa_p\textbf{x}\cdot\textbf{d}+\omega t)}
\end{equation}
which gives
\[|\nabla\times\textbf{E}^{inc}|=\sqrt{n}\omega.\]
Using the assumption about $g$ and the fact that $\textbf{f}$ is supported in $B_R$, we derive 
from (\ref{pf}) together with the previous two relations that
\ben
|\textbf{p}\cdot\widehat{\textbf{f}}(\kappa_p\textbf{d})\widehat{g}(\omega)|
&=&\Big|(2\pi)^{-2}\int_0^T\int_{\mathbb{R}^3} \textbf{p}\cdot\textbf{f}(\textbf{x})g(t)  e^{-i\kappa_p \textbf{x}\cdot\textbf{d}-i\omega t}\ \mathrm{d}\textbf{x}\mathrm{d}t\Big|\\
 &=&\Big|(2\pi)^{-2}\int_0^T\int_{B_R} \textbf{p}\cdot\textbf{f}(\textbf{x})g(t) \textbf{E}^{inc}(\textbf{x},t)\ \mathrm{d}\textbf{x}\mathrm{d}t\Big|\\
&\leq &C(1+\omega)\epsilon,
\enn
where $\epsilon=\Big(\int_0^T\int_{\partial B_R}\big(  |T(\textbf{E}\times\nu)|^2+|\textbf{E}\times\nu|^2 \big)ds(\textbf{x})dt\Big)^{\frac{1}{2}}$ represents the measurement data on $\partial B_R\times (0, T)$.
In view of the assumption \eqref{af}, one obtains for $\omega\in(0, b)$, $b>1$ and  $\kappa_p^2=n\omega^2$ that
\begin{equation}\label{phatf}
\begin{split}
 |\textbf{p}\cdot\widehat{\textbf{f}}(\kappa_p\textbf{d})|\leq C\frac{\omega\epsilon}{|\hat{g}(\omega)|}&\leq C\,\delta^{-1}(1+\omega)\epsilon\\
&\leq Cb\epsilon,
\end{split}
\end{equation}
where $C>0$ depends on $n$, $\delta$, $T_0$ and $R$.

Repeating similar steps, we get  
\be
|\textbf{q}\cdot\hat{\textbf{f}}(\kappa_p\textbf{d})|\leq C\,b\,\epsilon \quad \mbox{for all} \  |\kappa_p|<\sqrt{n}b.
\en
On the other hand, Since $\nabla\cdot\textbf{f}=0$, we obtain that 
\begin{equation}
-i\kappa_p\textbf{d}\cdot\widehat{\textbf{f}}(-\kappa_p\textbf{d})=-i\kappa_p\int_{B_R}\textbf{d}e^{-i\kappa_p\textbf{x}\cdot\textbf{d}}\cdot\textbf{f}(\textbf{x})d\textbf{x}=\int_{B_R}\nabla e^{-i\kappa_p\textbf{x}\cdot\textbf{d}}\cdot\textbf{f}(\textbf{x})d\textbf{x}=\int_{B_R}e^{-i\kappa_p\textbf{x}\cdot\textbf{d}}\nabla\cdot \textbf{f}(\textbf{x})d\textbf{x}=0,
\end{equation}
which means $\textbf{d}\cdot\widehat{\textbf{f}}(-\kappa_p\textbf{d})=0$.
Using the Pythagorean theorem yields 
\begin{equation}
|\widehat{\textbf{f}}(\kappa_p\textbf{d})|^2=|\textbf{p}\cdot\widehat{\textbf{f}}(\kappa_p\textbf{d})|^2+|\textbf{q}\cdot\widehat{\textbf{f}}(\kappa_p\textbf{d})|^2+|\textbf{d}\cdot\widehat{\textbf{f}}(\kappa_p\textbf{d})|^2
=|\textbf{p}\cdot\widehat{\textbf{f}}(\kappa_p\textbf{d})|^2+|\textbf{q}\cdot\widehat{\textbf{f}}(\kappa_p\textbf{d})|^2.
\end{equation}
Let $\xi_p=\kappa_p\textbf{d}$, we obtain from the Parseval theorem that  
\begin{equation}\label{par1}
\|\textbf{f}\|_{L^2(\R^3)^3}^2=\|\widehat{\textbf{f}}\|_{L^2(\R^3)^3}^2=\int_{\R^3}|\widehat{\textbf{f}}(\xi_p)|^2d\xi_p
=\int_{\R^3}|\textbf{p}\cdot\widehat{\textbf{f}}(\xi_p)|^2d\xi_p+\int_{\R^3}|\textbf{q}\cdot\widehat{\textbf{f}}(\xi_p)|^2d\xi_p.
\end{equation}
Denote
\begin{equation}\nonumber
I(s)=\int_{|\xi_p|\leq \sqrt{n}s}|\textbf{p}\cdot\widehat{\textbf{f}}(\xi_p)|^2d\xi_p.
\end{equation}
Since the integrand $I(s)$ is an entire analytic function of $\xi$, the integral $I(s)$ with respect to $\xi_p$ can be taken over by any path joining points $0$ and $\sqrt{n}s$ in complex plane. Thus $I(s)$ is an entire analytic function of $s=s_1+is_2$ $(s_1,s_2\in \mathbb{R})$ and the following estimate holds.
\begin{lemma}\label{1}
Let $\textbf{f}(\textbf{x})\in L^2(\mathbb{R}^3)^3$, $\mathrm{supp}\ \textbf{f}\subset B_R $. Then
\begin{equation}\label{esI}
| I(s)|\leq (\frac{4\pi}{3})^2\,R^3\,\sqrt{n}^3|s|^3e^{2R\sqrt{n}|s_2|}||\textbf{f}||^2_{L^2(\mathbb{R}^3)^3},\qquad s=s_1+i s_2\in \C.
 \end{equation}
\end{lemma}
\begin{proof}
Set $l=\sqrt{n}sl'$ for $l'\in (0,1)$. Then it is easy to get
\begin{align*}
I(s)&\leq\int_{|\xi_p|\leq \sqrt{n}s}|\widehat{\textbf{f}}|^2(\xi_p)d\xi_p\\
&=\int_{0}^{\sqrt{n}s}\int_{\mathbb{S}^2} |\widehat{\textbf{f}}(l\hat{\theta})|^2 l^2 \mathrm{d}\hat{\theta}\mathrm{d} l\\
&=\int_{0}^{1}\int_{\mathbb{S}^2}|\widehat{\textbf{f}}(c_nsl'\hat{\theta})|^2 (\sqrt{n}s)^3l'^2 \mathrm{d}\hat{\theta}\mathrm{d}l'.
\end{align*}
Noting the elementary inequality $|e^{-i\sqrt{n}sl'\hat{\theta}\cdot\textbf{x}}|\leq e^{\sqrt{n}|s_2|R}$ for all $x\in B_R$ and $\hat{\theta}\in \mathbb{S}^2$, we have
\begin{align*}
|I(s)|&\leq\left|\int_{0}^{1}\int_{\mathbb{S}^2} \left|\int_{B_R}\textbf{f}(\textbf{x})e^{-i\sqrt{n}sl'\hat{\theta}\cdot \textbf{x}}\mathrm{d}\textbf{x}\right|^2 (c_ns)^3l'^2 \mathrm{d}\hat{\theta}\mathrm{d}l'\right|\\
&\leq \frac{4\pi}{3}R^3\int_{0}^{1}\int_{\mathbb{S}^2}\sqrt{n}^3|s|^3 \left(\int_{B_R}|\textbf{f}(\textbf{x})|^2|e^{2\sqrt{n}R|s_2|}|\, \mathrm{d}\textbf{x}\right)\mathrm{d}\hat{\theta}\mathrm{d}l'\\
&\leq  (\frac{4\pi}{3})^2\,R^3\,\sqrt{n}^3\,|s|^3e^{2R\sqrt{n}|s_2|}||\textbf{f}||^2_{L^2(\mathbb{R}^3)^3}.
\end{align*}
This completes the Lemma \ref{1}.
\end{proof}

The following Lemma is essential to show the relation between $I(s)$ for $s\in(b,\infty)$ with $I(b).$ Its proof can be found in \cite{Cheng2016Increasing}.
\begin{lemma}\label{2}
  Let $J(z)$ be an analytic function in $S=\{z=x+iy\in \mathbb{C}:-\frac{\pi}{4}<\arg z<\frac{\pi}{4}\}$ and continuous in $\overline{S}$ satisfying
  \begin{equation}
  \nonumber
  \begin{cases}
    |J(z)| \leq \epsilon, \  & z\in (0,L], \\
    |J(z)| \leq V, \  & z\in S,\\
    |J(0)|  =0.
  \end{cases}
  \end{equation}
Then there exists a function $\mu(z)$ satisfying
\begin{equation}
  \nonumber
  \begin{cases}
   \mu (z)  \geq\frac{1}{2},\ \ &z\in (L,2^{\frac{1}{4}}L), \\
   \mu (z)  \geq\frac{1}{\pi}((\frac{z}{L})^4-1)^{-\frac{1}{2}},\ \ & z\in (2^{\frac{1}{4}}L,\infty)
  \end{cases}
\end{equation}
   such that
\begin{equation}
\nonumber
 |J(z)|\leq V\epsilon^{\mu(z)}, \ \ \forall z\in (L,\infty).
\end{equation}
\end{lemma}
Let the sector $S\subset\C$ be given as in Lemma  \ref{2}.
Now, it follows from Lemma \ref{1} that
\begin{equation}
\nonumber
|I(s)e^{-(2R\sqrt{n}+1)s}|\leq CM^2 \quad \mbox{for all}\quad s\in S,
\end{equation}
where $C>0$ depends on $n$ and $R$.
Recalling from a priori estimate \eqref{phatf}, we obtain 
\begin{equation}
\nonumber
I(s)=\int_{|\xi_p|\leq \sqrt{n}s}|\textbf{p}\cdot\widehat{\textbf{f}}(\xi_p)|^2d\xi_p\leq C\,s^3\,b^2\,\epsilon^2, \quad s\in(0,b]
\end{equation}
where $C>0$ depends on $n$, $T_0$ and $R$.
Hence
\begin{equation}
\nonumber
| I(s)e^{-(2R\sqrt{n}+1)s}|\leq C b^2\,\epsilon^2, \ \ s\in(0,b].
\end{equation}
Then applying Lemma \ref{2} with $L=b$ to the function $J(s):=\frac{1}{b^2}I(s)e^{-(2R\sqrt{n}+1)s}$ , we know that there exists a function $\mu(s)$ satisfying
  \begin{equation}\nonumber
  \begin{cases}
  \mu (s)  \geq\frac{1}{2},\ \ &s\in (b,2^{\frac{1}{4}}b), \\
   \mu (s)  \geq\frac{1}{\pi}((\frac{s}{b})^{4}-1)^{-\frac{1}{2}},\ \ &s\in (2^{\frac{1}{4}}b,\infty)
  \end{cases}
\end{equation}
such that
  \begin{equation}
   \nonumber
   |\frac{1}{b^2}I(s)e^{-(2R\sqrt{n}+1)s}|\leq CM^2\epsilon^{2\mu} \ \ \quad\mbox{for all}\quad s\in(b,\infty),
   \end{equation}
where $C>0$ depends on $n$, $T_0$ and $R$.

Now we show the proof of Theorem \ref{TH1}. We assume that $\epsilon<e^{-1}$, otherwise the estimate is obvious. 
Let
\begin{equation}
\nonumber
s= \begin{cases}
  \frac{1}{((2R\sqrt{n}+3)\pi)^{\frac{1}{3}}}b^{\frac{2}{3}}|\ln \epsilon|^{\frac{1}{4}},  \qquad&\mbox{if} \quad 2^{\frac{1}{4}} ((2R\sqrt{n}+3)\pi)^{\frac{1}{3}}b^{\frac{1}{3}}< |\ln \epsilon|^{\frac{1}{4}},\\
     b,  \qquad&\mbox{if} \quad |\ln \epsilon|^{\frac{1}{4}}\leq 2^{\frac{1}{4}}((2R\sqrt{n}+3)\pi)^{\frac{1}{3}} b^{\frac{1}{3}}.
  \end{cases}
\end{equation}
Case (i):  $2^{\frac{1}{4}}((2R\sqrt{n}+3)\pi)^{\frac{1}{3}} b^{\frac{1}{3}}< |\ln \epsilon|^{\frac{1}{4}}$. Then we have
\begin{eqnarray*}
 |I(s)|&\leq& CM^2b^2\epsilon^{2\mu}e^{(2R\sqrt{n}+1)s}\\
 &=&CM^2b^2e^{(2R\sqrt{n}+1)s-2\mu(s)|\ln\epsilon|}\\
           &\leq& CM^2b^2e^{\frac{(2R\sqrt{n}+3)}{((2R+3)\pi)^{\frac{1}{3}}}b^{\frac{2}{3}}|\ln \epsilon|^{\frac{1}{4}}-\frac{2|\ln \epsilon|}{\pi}(\frac{b}{s})^2 }\\
           &=& CM^2b^2 e^{-2\big(\frac{(2R\sqrt{n}+3)^2}{\pi}\big)^{\frac{1}{3}}b^{\frac{2}{3}}|\ln \epsilon|^{\frac{1}{2}}(1- \frac{1}{2}|\ln \epsilon|^{-\frac{1}{4}}) }, \\
\end{eqnarray*}
where $C>0$ depends on $n$, $T_0$ and $R$.
Noting that $|\ln\epsilon|^{-\frac{1}{4}}<1$ and $\big(\frac{(2R\sqrt{n}+3)^2}{\pi}\big)^{\frac{1}{3}}>1$, we obtain
\[|I(s)|\leq CM^2b^2e^{-b^{\frac{2}{3}}|\ln \epsilon|^{\frac{1}{2}}}.\]
Using the inequality $e^{-t}\leq\frac{6!}{t^6}$ for $t>0$, we get
\begin{equation}\label{ess}
  |I(s)|\leq C\frac{M^2}{b^2|\ln \epsilon|^{3}}.
\end{equation}
Since $b^2|\ln \epsilon|^{3}\geq b^{\frac{4}{3}}|\ln \epsilon|^{\frac{1}{2}}$ when $b>1$ and $|\ln\epsilon|>1$.
Hence
\begin{equation}\label{ess1}
  \begin{split}
  I(s)+\int_{|\xi_p|>\sqrt{n}s} |\textbf{p}\cdot\widehat{\textbf{f}}(\xi_p)|^2\,\mathrm{d}\xi_p&\leq I(s)+\frac{M^2}{n^2s^2}\\
                         &\leq C(\frac{ M^2}{b^4|\ln \epsilon|^{3}}+\frac{ M^2}{b^{\frac{4}{3}}|\ln \epsilon|^{\frac{1}{2}}})\\
                         &\leq \frac{ CM^2}{b^{\frac{4}{3}}|\ln \epsilon|^{\frac{1}{2}}},
  \end{split}
\end{equation}
where $C>0$ depends on $n$, $T_0$ and $R$.

Case (ii):  $|\ln \epsilon|^{\frac{1}{4}}\leq 2^{\frac{1}{4}}((2R\sqrt{n}+3)\pi)^{\frac{1}{3}} b^{\frac{1}{3}}$. In this case we have $s=b$ by the choice of $s$ and $|I(b)|\leq C\,b^5\epsilon^2$. Hence,
 \begin{equation}\label{ess2}
  \begin{split}
     I(b)+\int_{|\xi_p|>\sqrt{n}b} |\textbf{p}\cdot\widehat{\textbf{f}}(\xi_p)|^2\,\mathrm{d}\xi_p\leq C(b^5\epsilon^2+\frac{ M^2}{b^2})\\
 \leq C(b^5\epsilon^2+\frac{M^2}{b^{\frac{4}{3}}|\ln \epsilon|^{\frac{1}{2}}}),
  \end{split}
 \end{equation}
 where $C>0$ depends on $n$, $\delta$, $T_0$ and $R$.
Combining (\ref{ess1}) and (\ref{ess2}), we finally obtain
\begin{equation*}
\int_{\R^3} |\textbf{p}\cdot\widehat{\textbf{f}}(\xi_p)|^2d\xi_p  \leq  C(b^5\epsilon^2+\frac{M^2}{b^{\frac{4}{3}}|\ln \epsilon|^{\frac{1}{2}}}).
\end{equation*}
Repeating similar steps by multiplying $\textbf{H}^{inc}$ on both sides of the equation (\ref{eq1}), we get 
\begin{equation*}
\int_{\R^3} |\textbf{q}\cdot\widehat{\textbf{f}}(\xi_p)|^2d\xi_p \leq  C(b^5\epsilon^2+\frac{M^2}{b^{\frac{4}{3}}|\ln \epsilon|^{\frac{1}{2}}}),
\end{equation*}
where $C>0$ depends on $n$, $\delta$, $T_0$ and $R$.
Combining the above estimates and (\ref{par1}), 
we complete the proof.

\section{Proof of Theorem \ref{TH3}}\label{th3}
In this section. Let $n\in(0, b^2)$ for $b>1$. Through this section, $C>0$ denotes a generic constant which is independent of $n$, may vary from line to line. Consider the equation
\begin{equation}\label{F}
\begin{cases}
n\partial_t^2\textbf{E}(\textbf{x},t)+\nabla\times(\nabla\times \textbf{E}(\textbf{x},t))=\textbf{F}(\textbf{x},t), \ &(\textbf{x},t)\in\R^3\times(0,\infty),\\
\textbf{E}(\textbf{x},0)=0, \quad \partial_t\textbf{E}(\textbf{x},0)=0, \ &\textbf{x}\in\R^3.
\end{cases}
\end{equation}
Assume that $\textbf{E}(\textbf{x},t):=\textbf{E}(\textbf{x},t;n)$.

Multiplying the both sides of (\ref{F}) by $\textbf{E}^{inc}$ and using the integration by parts over $B_R$, we obtain
\begin{equation}\label{einc}
  \int_0^T\int_{B_R} (n\partial_t^2\textbf{E}+\nabla\times(\nabla\times\textbf{E})\textbf{E}^{inc}\ \mathrm{d}\textbf{x}\mathrm{d}t=\int_0^T\int_{B_R} \textbf{F}(\textbf{x},t) \textbf{E}^{inc}(\textbf{x},t)\ \mathrm{d}\textbf{x}\mathrm{d}t.
\end{equation}
Integrating by parts, one deduces from the left hand side of (\ref{einc}) that
\begin{align*}
  \int_0^T\int_{B_R}&(n\partial_t^2\textbf{E}+\nabla\times(\nabla\times\textbf{E})\textbf{E}^{inc}\ \mathrm{d}\textbf{x}\mathrm{d}t \\
  =&\int_{B_R} n(\partial_{t}\textbf{E}(\textbf{x},t) \textbf{E}^{inc}(\textbf{x},t)- \textbf{E}(\textbf{x},t) \partial_{t}\textbf{E}^{inc}(\textbf{x},t))\big|_0^T\;\mathrm{d}\textbf{x}\\
  &\quad+\int_0^T\int_{\partial B_R} \big(\nu\times(\nabla\times \textbf{E})\cdot\textbf{E}^{inc}-\nu\times(\nabla\times \textbf{E}^{inc})\cdot \textbf{E}\big) ds(\textbf{x})dt.\\
=&-\int_0^T\int_{\partial B_R} \big((T(\textbf{E}\times\nu)\cdot \textbf{E}^{inc})+(\textbf{E}\times\nu)\cdot(\nabla\times \textbf{E}^{inc})\big)ds(\textbf{x})dt.
\end{align*}
Note that, in the last step we have used the fact that $\textbf{E}(\textbf{x},t)=0$ when $|\textbf{x}|<R$ and $t>T_0+2R$, which follows straightforwardly from Huygens' principle (see e.g.,
 \cite[Lemma 2.1]{BaoHu}). This implies $\textbf{E}(\textbf{x},T)=\partial_t \textbf{E}(\textbf{x},T)=0$ for $\textbf{x}\in B_R$ and $T>T_0+2R$. Hence, the integral over $B_R$ on the left hand side of the previous identity vanishes. Thus
\begin{equation}
\int_{B_R}\int_0^T\textbf{p}e^{-i(\xi_p\cdot\textbf{x}+\omega t)}\cdot\textbf{F}(\textbf{x},t)d\textbf{x}dt=
-\int_0^T\int_{\partial B_R}\big((T(\textbf{E}\times\nu)\cdot \textbf{E}^{inc})+(\textbf{E}\times\nu)\cdot(\nabla\times \textbf{E}^{inc})\big)ds(\textbf{x})dt.
\end{equation}
Define
\begin{equation}
\nonumber
\widehat{\textbf{F}}(\xi,\omega)=(2\pi)^{-2}\int_{\R^4}\textbf{F}(\textbf{x},t)e^{-i(\xi\cdot \textbf{x}+\omega t)}d\textbf{x}dt,
\end{equation}
with $(\xi,\omega)=(\xi_1,\xi_2,\xi_3,\omega)\in\R^4$.
Since supp $\textbf{F}(\textbf{x},t)$ $\subset B_R\times(0,T_0)$, we have
\begin{equation}\label{pF}
\begin{split}
(2\pi)^2\textbf{p}\cdot\widehat{\textbf{F}}(\xi_p,\omega)
&=\int_0^T\int_{B_R}\textbf{p}\cdot\textbf{F}(\textbf{x},t)e^{-i(\xi_p\cdot \textbf{x}+\omega t)}d\textbf{x}dt \\
& = -\int_0^T\int_{\partial B_R}\big((T(\textbf{E}\times\nu)\cdot \textbf{E}^{inc})+(\textbf{E}\times\nu)\cdot(\nabla\times \textbf{E}^{inc})\big)ds(\textbf{x})dt.
\end{split}
\end{equation}
Consider the set (see Figure 1)
$$E_s=\{(\xi,\omega)\in \R^4|\ |\xi|^2-n\omega^2=0,\ |\omega|<s,\   n \in(0,s^2)\}.$$ 

\begin{figure}[htbp]
 \centering	
\includegraphics[width=0.48\textwidth]{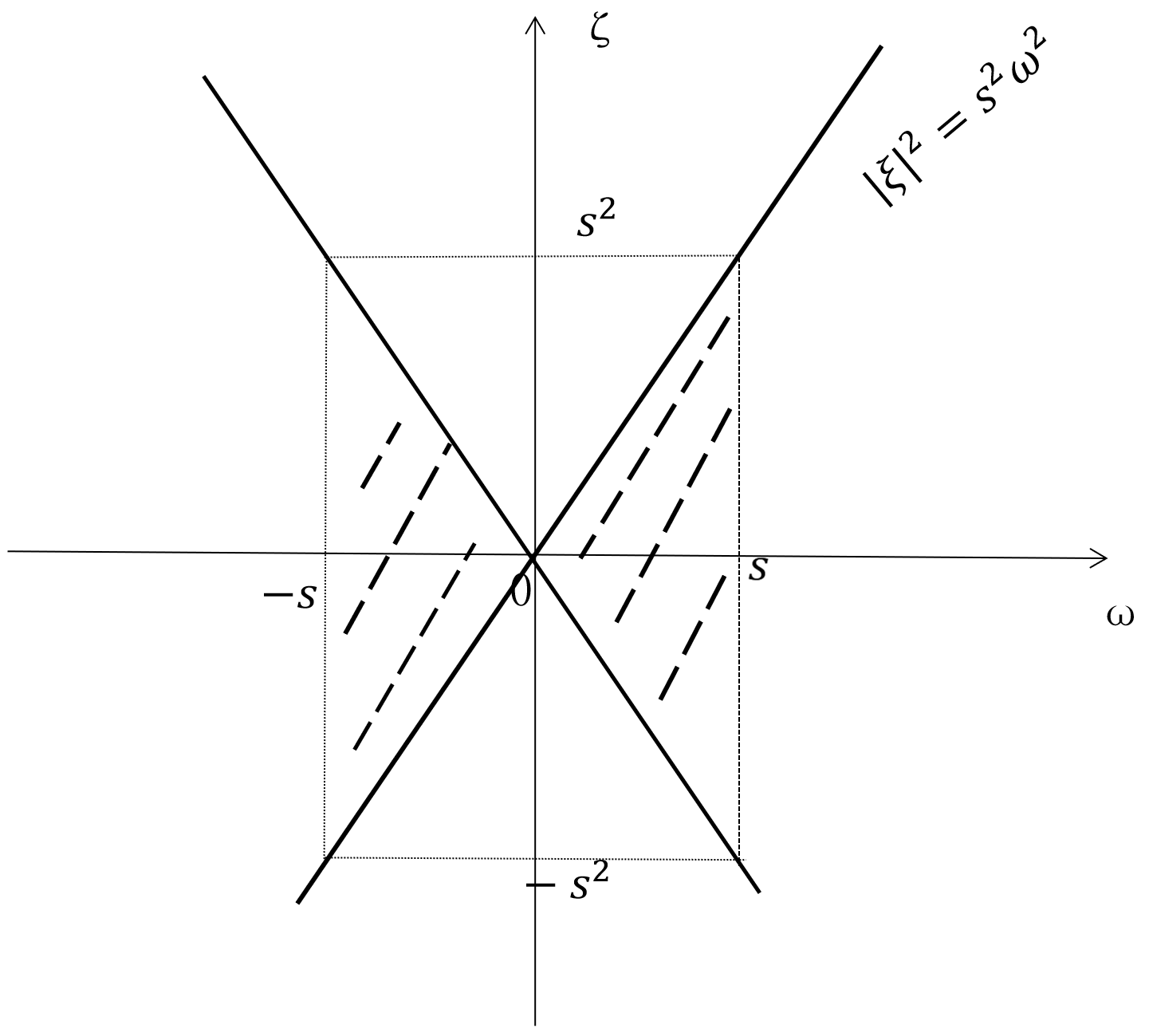}
	\caption{$E_s$ is the shadow area, $\zeta^2=|\xi|^2$.}
	\label{Fig.main}
\end{figure}
Combining $(\ref{pF})$, we have
\begin{equation}\label{hatH}
|\textbf{p}\cdot\widehat{\textbf{F}}(\xi_p,\omega)|\leq Cb^2\epsilon, \quad (\xi_p,\omega)\in {E_{b}}
\end{equation}
where $\epsilon=\sup\limits_{n\in(0,b^2)}\Big(\int_0^T\int_{\partial B_R}\big(|T(\textbf{E}\times\nu)|^2+|\textbf{E}\times\nu|^2 \big)ds(\textbf{x})dt\Big)^{\frac{1}{2}}$.
Denote
\begin{equation}\label{pfb}
I(s):=\int_{E_s}|\textbf{p}\cdot\widehat{\textbf{F}}(\xi_p,\omega)|^2d\xi_p\,d\omega.
\end{equation}
It is easy to verify from (\ref{pfb}) that
\begin{align}\label{IE}
|I(s)|\leq C\,|E_{s}|\,b^4\,\epsilon^2, \quad s\in[0,b]
\end{align}
and 
\begin{equation}
I(s)\leq\int_{E_s}|\widehat{\textbf{F}}(\xi_p,\omega)|^2d\xi_p\,d\omega.
\end{equation}
Using the polar coordinates, we deduce that 
\[
\int_{E_s\cap\{(\xi,\omega)|\ \omega\geq0\}}|\widehat{\textbf{F}}(\xi,\omega)|^2d\xi d\omega
=\int_0^s\big(\int_0^{2\pi}\int_0^{\pi}\int_0^{s\omega}|\widehat{\textbf{F}}(r\hat{\xi},\omega)|^2 r^2\sin\varphi dr d\theta d\varphi\big)d\omega.\]
Let $r=s\omega\hat{r}$ and $\omega=s\hat{\omega}$ for $\hat{r},\hat{\omega}\in(0,1)$. 
A simple calculation yields 
\begin{equation}
\nonumber
\begin{split}
\int_{E_s\cap\{(\xi,\omega)|\ \omega\geq0\}}&|\widehat{\textbf{F}}(\xi,\omega)|^2d\xi d\omega
=\int_0^s\big(\int_0^{2\pi}\int_0^{\pi}\int_0^{1}|\widehat{\textbf{F}}(s\omega\hat{r}\hat{\xi},\omega)|^2 (s\omega)^3\hat{r}^2\sin\varphi dr d\theta d\varphi\big)d\omega\\
&=\int_0^1\big(\int_0^{2\pi}\int_0^{\pi}\int_0^{1}|\widehat{\textbf{F}}(s^2\omega\hat{r}\hat{\xi},\omega)|^2 s(s^2\hat{\omega})^3\hat{r}^2\sin\varphi dr d\theta d\varphi\big)d\hat{\omega}\\
&=\int_0^1\big(\int_0^{2\pi}\int_0^{\pi}\int_0^{1}| \int_{\R^4}\widehat{\textbf{F}}(x,t)e^{-i(s^2\hat{\omega}\hat{r}\hat{\xi}\cdot x+s\hat{\omega}t)}dxdt|^2 s(s^2\hat{\omega})^3\hat{r}^2\sin\varphi dr d\theta d\varphi\big)d\hat{\omega}.
\end{split}
\end{equation}
This integrals are analytic functions of $s=s_1+is_2$, $s_1,s_2\in\R$. 
Noting that $e^{-i(s^2\hat{\omega}\hat{r}\hat{\xi}\cdot x+s\hat{\omega}t)}\leq e^{2R|s_1s_2|+T|s_2|}$.
We have for all $s\in S$ that 
\begin{equation}
\nonumber
\big|\int_{E_s\cap\{(\xi,\omega)|  \omega\geq0\}}|\widehat{\textbf{F}}(\xi,\omega)|^2d\xi d\omega\big|\leq CM^2|s|^6e^{2R|s_1s_2|+T|s_2|}.
\end{equation}
Hence we also obtain
\begin{equation}
\nonumber
|I(s)|\leq C\,M^2|s|^6e^{2R|s_1s_2|+T|s_2|}, \quad s\in S.
\end{equation}
Let $\Delta=\max\{2R, T\}$, it follows for all $s\in S$ that
\begin{equation}
\nonumber
|e^{-(\Delta+1)s-(\Delta+1)s^2}I(s)|\leq CM^2.
\end{equation}
Recalling from (\ref{IE}) a prior estimate, we obtain
\begin{equation}
\nonumber
|e^{-(\Delta+1)s-(\Delta+1)s^2}I(s)|\leq C\,b^4\,\epsilon^2 \quad \mbox{for all} \quad s\in[0,b].
\end{equation}
Then applying Lemma \ref{2} with $L=b$ to the function $\frac{1}{b^4}e^{-(\Delta+1)s-(\Delta+1)s^2}I(s)$, we know that 
\begin{equation}
\nonumber
|I(s)|\leq CM^2\,b^4\,e^{2(\Delta+1)s^2}\epsilon^{2\mu}, \quad \forall s\in(b,\infty).
\end{equation}
Consider the set $$E^1_{s}=\{(\xi,\omega)\in \R^4| \ \ |\omega|\geq s,\ |\xi|\leq s|\omega|\},$$
we have by using the Parseval's identity that 
\begin{equation}
\nonumber
\begin{split}
I_1(s):=\int_{E^1_{s}}|\textbf{p}\cdot\widehat{\textbf{F}}|^2d\xi_p\,d\omega\leq\frac{1}{s^2}\int_{\R^3}|\widehat{\nabla (\textbf{p}\cdot\textbf{F})}|^2d\xi_p\,d\omega
\leq\,C\,\frac{M^2}{s^2}.
\end{split}
\end{equation}
Denote
\begin{equation}
\nonumber
I_2(s):=\int_{E^2_{s}}|\textbf{p}\cdot\widehat{\textbf{F}}|^2d\xi_p\,d\omega,
\end{equation}
where
$$E^2_{s}=\{(\xi,\omega)\in \R^4|\  \ |\xi|\geq s|\omega| \}.$$
It is easy to obtain
\begin{equation}
\nonumber
I_2(s)\leq\int_{E^2_{s}}|\widehat{\textbf{F}}|^2d\xi\,dw,
\end{equation}
Similarly, by using the polar coordinates. Let $r=s\hat{r}$ and $\omega=s^2\hat{r}\hat{\omega}$ for $\hat{r},\hat{\omega}\in(0,1)$. We get
\begin{equation}
\nonumber
\begin{split}
&\int_{E^2_{s}\cap\{(\xi,\omega)|\ \omega\geq0\}}|\textbf{p}\cdot\widehat{\textbf{F}}(\xi,\omega)|^2d\xi d\omega\\
&=\int_0^{2\pi}\int_0^{\pi}\int_0^{\infty}\big(\int_0^{\frac{r}{s}}|\textbf{p}\cdot\widehat{\textbf{F}}(r\hat{\xi},\omega)|^2d\omega\big)r^2\sin\varphi dr\,d\theta\,d\varphi\\
&=\int_0^{2\pi}\int_0^{\pi}\int_0^{\infty}\big(\int_0^{1}|\textbf{p}\cdot\widehat{\textbf{F}}(r\hat{\xi},\frac{1}{s}r\hat{\omega})|^2\frac{1}{s}rd\hat{\omega}\big)r^2\sin\varphi dr\,d\theta\,d\varphi\\
&=\int_0^{2\pi}\int_0^{\pi}\int_0^{1}\big(\int_0^{1}|\textbf{p}\cdot\widehat{\textbf{F}}(r\hat{\xi},\frac{1}{s}r\hat{\omega})|^2\frac{1}{s}rd\hat{\omega}\big)r^2\sin\varphi dr\,d\theta\,d\varphi\\
&+\int_0^{2\pi}\int_0^{\pi}\int_1^{\infty}\big(\int_0^{1}|\textbf{p}\cdot\widehat{\textbf{F}}(r\hat{\xi},\frac{1}{s}r\hat{\omega})|^2\frac{1}{s}rd\hat{\omega}\big)r^2\sin\varphi dr\,d\theta\,d\varphi.
\end{split}
\end{equation}
Since 
\begin{equation}
\nonumber
\begin{split}
\int_0^{1}|\textbf{p}\cdot\widehat{\textbf{F}}(r\hat{\xi},\frac{1}{\sqrt{n}}r\hat{\omega})|^2d\hat{\omega}
=\frac{1}{r^6}\int_0^{1} |\int_{\R^4}\textbf{p}\cdot\Delta\textbf{F}(x,t)e^{-i(r\hat{\xi}\cdot x+\frac{1}{\sqrt{n}}r\hat{\omega} t)}dxdt|^2d\hat{\omega}
\leq \frac{M^2}{r^3},
\end{split}
\end{equation}
which gives
\begin{equation}
\nonumber
\big|\int_{E^2_{s}\cap\{(\xi,\omega)|\  \omega\geq0\}}|\textbf{p}\cdot\widehat{\textbf{F}}(\xi,\omega)|^2d\xi d\omega\big| \leq C\,\frac{M^2}{s}.
\end{equation}
Consequently
\begin{equation}
\nonumber
I_2(s) \leq C\,\frac{M^2}{s}.
\end{equation}

Now we show the proof of Theorem $\ref{TH3}$.
We assume that $\epsilon<e^{-1}$, otherwise the estimate is obvious. Let 
\begin{eqnarray}
s=
\begin{cases}
\frac{1}{(2(\Delta+2)\pi)^{\frac{1}{4}}}b^{\frac{1}{2}}|\ln\epsilon|^{\frac{1}{5}} & \textrm{if}\ |\ln\epsilon|^{\frac{1}{5}}>2^{\frac{1}{4}}b^{\frac{1}{2}}(2(\Delta+2)\pi)^{\frac{1}{4}},\\
b & \textrm{if} \ |\ln\epsilon|^{\frac{1}{5}}\leq2^{\frac{1}{4}}b^{\frac{1}{2}}(2(\Delta+2)\pi)^{\frac{1}{3}}.
\end{cases}
\end{eqnarray}

Case (i): $|\ln\epsilon|^{\frac{1}{5}}>2^{\frac{1}{4}}b^{\frac{1}{2}}(2(\Delta+2)\pi)^{\frac{1}{4}}$. 
One can check that 
\[s>2^{\frac{1}{4}}b.\]
Thus, using Lemma \ref{2}, we obtain
\begin{eqnarray}
|I(s)|&\leq & C\,M^2\,b^4\, e^{2(\Delta+2)s^2}\epsilon^{2\mu(s)}\cr
&\leq & C\,M^2\,b^4\, e^{-2\mu(s)|\ln\epsilon|+2(\Delta+2)s^2}\cr
&\leq & C\,M^2\,b^4\,e^{-\big(\frac{-2|\ln\epsilon|}{\pi}(\frac{b}{s})^2+2(\Delta+2)s^2\big)}\cr
&\leq & C\,M^2\,b^4\,e^{-2(\frac{2(\Delta+2)}{\pi})^{\frac{1}{2}}b|\ln\epsilon|^{\frac{3}{5}}(1-\frac{1}{2}|\ln\epsilon|^{-\frac{1}{5}}) }.
\end{eqnarray}
Noting that 
$\frac{1}{2}|\ln\epsilon|^{-\alpha}<\frac{1}{2}$ and $(\frac{2(\Delta+2)}{\pi})^{\frac{1}{2}}>1$,
we have 
\[|I(s)|\leq C\,M^2\,b^4\,e^{-b|\ln\epsilon|^{\frac{3}{5}}}.\]
Using the elementary inequality 
\[e^{-t}\leq \frac{5!}{t^5}, \quad t>0,\]
we get 
\begin{equation}
\nonumber
|I(s)|\leq C\frac{b^4M^2}{b^5|\ln\epsilon|^{3}}  \leq C\frac{M^2}{b|\ln\epsilon|^{3}}.
\end{equation}

Case (ii): $|\ln\epsilon|^{\frac{1}{5}}\leq2^{\frac{1}{4}}b^{\frac{1}{2}}(2(\Delta+2)\pi)^{\frac{1}{4}}$. In this case we have for $s=b$ that
\[|I(s)|=|I(b)| \leq b^{11}\,\epsilon^2.\]
Combining  estimates of $I_1(s)$ and $I_2(s)$ , we obtain 
\begin{eqnarray}\label{PF1}
\int_{E_s}|\textbf{p}\cdot\widehat{\textbf{F}}|^2d\xi_p\,d\omega+\int_{E^1_s}|\textbf{p}\cdot\widehat{\textbf{F}}|^2d\xi_p\,d\omega
+\int_{E^2_s}|\textbf{p}\cdot\widehat{\textbf{F}}|^2d\xi_p\,d\omega
&=&I(s)+I_1(s)+I_2(s)\cr
&\leq& C\Big(b^{11}\epsilon^2+\frac{M^2}{b|\ln\epsilon|^{3}}+\frac{M^2}{b^{\frac{1}{2}}|\ln\epsilon|^{\frac{1}{5}}}  \Big)\cr
&\leq& C\Big(b^{11}\epsilon^2+\frac{M^2}{b^{\frac{1}{2}}|\ln\epsilon|^{\frac{1}{5}}}\Big)
\end{eqnarray}
Since $\big(b|\ln\epsilon|^{3}\big)>\big(b|\ln\epsilon|^{\frac{1}{5}}\big)$ when $b>1$ and $|\ln\epsilon|>1$.

Repeating similar steps, we get 
\begin{equation}\label{PF2}
\int_{E_s}|\textbf{q}\cdot\widehat{\textbf{F}}|^2d\xi_p\,d\omega+\int_{E^1_s}|\textbf{q}\cdot\widehat{\textbf{F}}|^2d\xi_p\,d\omega
+\int_{E^2_s}|\textbf{q}\cdot\widehat{\textbf{F}}|^2d\xi_p\,d\omega
\leq C\Big(b^{11}\epsilon^2+\frac{M^2}{b^{\frac{1}{2}}|\ln\epsilon|^{\frac{1}{5}}}\Big).
\end{equation}
On the other hand, since $\nabla\cdot\textbf{F}=0$, we have that 
\[-i\kappa_p\textbf{d}\cdot\widehat{\textbf{F}}(\kappa_p\textbf{d},\omega)=\int_{\R^4}\nabla e^{-i(\kappa_p \textbf{x}\cdot\textbf{d}+\omega t)}\cdot \textbf{F}(\textbf{x},t)d\textbf{x}dt=\int_0^T\int_{B_R} e^{-i(\kappa_p \textbf{x}\cdot\textbf{d}+\omega t)}\nabla\cdot \textbf{F}(\textbf{x},t)d\textbf{x}dt=0\]
Using the Pythagorean theorem yields
\begin{equation}
\nonumber
|\widehat{\textbf{F}}(\xi_p,\omega)|^2=|\textbf{p}\cdot\widehat{\textbf{F}}(\xi_p,\omega)|^2
+|\textbf{q}\cdot\widehat{\textbf{F}}(\xi_p,\omega)|^2+|\textbf{d}\cdot\widehat{\textbf{F}}(\xi_p,\omega)|^2.
\end{equation}
According to $\R^4=E_s\cup E^1_s\cup E^2_s$, we get
\begin{equation}\label{PF3}
\begin{split}
||\textbf{F}||^2_{L^2(\R^4)^3}=||\widehat{\textbf{F}}||^2_{L^2(\R^4)^3}
&=\int_{E_s}|\textbf{p}\cdot\widehat{\textbf{F}}|^2d\xi_p\,d\omega+\int_{E^1_s}|\textbf{p}\cdot\widehat{\textbf{F}}|^2d\xi_p\,d\omega
+\int_{E^2_s}|\textbf{p}\cdot\widehat{\textbf{F}}|^2d\xi_p\,d\omega\\
&+\int_{E_s}|\textbf{q}\cdot\widehat{\textbf{F}}|^2d\xi_p\,d\omega+\int_{E^1_s}|\textbf{q}\cdot\widehat{\textbf{F}}|^2d\xi_p\,d\omega
+\int_{E^2_s}|\textbf{q}\cdot\widehat{\textbf{F}}|^2d\xi_p\,d\omega.
\end{split}
\end{equation}
Combining the above estimates (\ref{PF1})-(\ref{PF3}), we obtain the stability estimate (\ref{es2}).
This completes the proof.

\section{Proof of Theorem \ref{TH4}}\label{th4}

In this section, we assume that $\textbf{F}$ takes the form
\begin{equation}\label{3F0}
\textbf{F}(x_1,x_2,x_3,t)=\textbf{f}(\tilde{x},t)g(x_3), \quad \tilde{x}=(x_1, x_2)\in\R^2, \ x_3\in\R, \quad   t\in(0,T_0).
\end{equation}
Then the equation $(\ref{eq1})$ becomes
\begin{equation}\label{fg}
n\partial_t^2\textbf{E}+\nabla\times(\nabla\times\textbf{E})=\textbf{f}(\tilde{x},t)\,g(x_3).
\end{equation}
Assume that $g$ is known, we establish an increasing stability estimate about $\textbf{f}$ from the Dirichlet data $\{\textbf{E}(\textbf{x},t)| \ \textbf{x}\in\partial B_R, \ t\in(0,T_0)\}$.

Multiplying the both sides of (\ref{fg}) by $\textbf{E}^{inc}$, we obtain
\begin{equation}\label{nfE}
  \int_0^T\int_{B_R} (n\partial_t^2\textbf{E}+\nabla\times(\nabla\times\textbf{E})\textbf{E}^{inc}\ \mathrm{d}\textbf{x}\mathrm{d}t=\int_0^T\int_{B_R} \textbf{f}(\tilde{x},t)\,g(x_3) \textbf{E}^{inc}(\textbf{x},t)\ \mathrm{d}\textbf{x}\mathrm{d}t.
\end{equation}
Integrating by parts, one deduces from the left hand side of (\ref{nfE}) that
\begin{align*}
  \int_0^T\int_{B_R}&(n\partial_t^2\textbf{E}+\nabla\times(\nabla\times\textbf{E})\textbf{E}^{inc}\ \mathrm{d}\textbf{x}\mathrm{d}t \\
  =&\int_{B_R} n(\partial_{t}\textbf{E}(\textbf{x},t) \textbf{E}^{inc}(\textbf{x},t)- \textbf{E}(\textbf{x},t) \partial_{t}\textbf{E}^{inc}(\textbf{x},t))\big|_0^T\;\mathrm{d}\textbf{x}\\
  &\quad+\int_0^T\int_{\partial B_R} \big(\nu\times(\nabla\times \textbf{E})\cdot\textbf{E}^{inc}-\nu\times(\nabla\times \textbf{E}^{inc})\cdot E\big) ds(\textbf{x})dt.\\
=&-\int_0^T\int_{\partial B_R} \big((T(\textbf{E}\times\nu)\cdot \textbf{E}^{inc})+(\textbf{E}\times\nu)\cdot(\nabla\times \textbf{E}^{inc})\big)ds(\textbf{x})dt.
\end{align*}
Note that, in the last step we have used the fact that $\textbf{E}(\textbf{x},t)=0$ when $|\textbf{x}|<R$ and $t>T_0+2R$, which follows straightforwardly from Huygens' principle (see e.g.,
 \cite[Lemma 2.1]{HLLZ2018}). This implies $\textbf{E}(\textbf{x},T)=\partial_t \textbf{E}(\textbf{x},T)=0$ for $\textbf{x}\in B_R$ and $T>T_0+2R$. Hence, the integral over $B_R$ on the left hand side of the previous identity vanishes. 
Thus (\ref{nfE}) becomes
\begin{equation}\label{pfg}
\begin{split}
\int_{B_R}\textbf{p}e^{-i(\kappa_pd_1x_1+\kappa_pd_2x_2+\omega t)}\cdot&\textbf{f}(x_1,x_2,t)d\tilde{x}\int_0^Tg(x_3)e^{-i\xi_3x_3}dt\\
&=-\int_0^T\int_{\partial B_R} \big((T(\textbf{E}\times\nu)\cdot \textbf{E}^{inc})+(\textbf{E}\times\nu)\cdot(\nabla\times \textbf{E}^{inc})\big)ds(\textbf{x})dt
\end{split}
\end{equation}
Let $\tilde{\xi}^2=\xi_1^2+\xi_2^2$.
Define the set (see Figure 2)
\begin{equation}
\nonumber
E_s=\{(\xi_1,\xi_2,\omega)\in\R^3| \ n\omega^2-|\xi_1|^2-|\xi_2|^2=|\xi_3|^2, \ |\omega|\leq\frac{s}{\sqrt{n}}, \ |\xi_3|\leq s\}.
\end{equation}
\begin{figure}[htbp]
 \centering	
\includegraphics[width=0.48\textwidth]{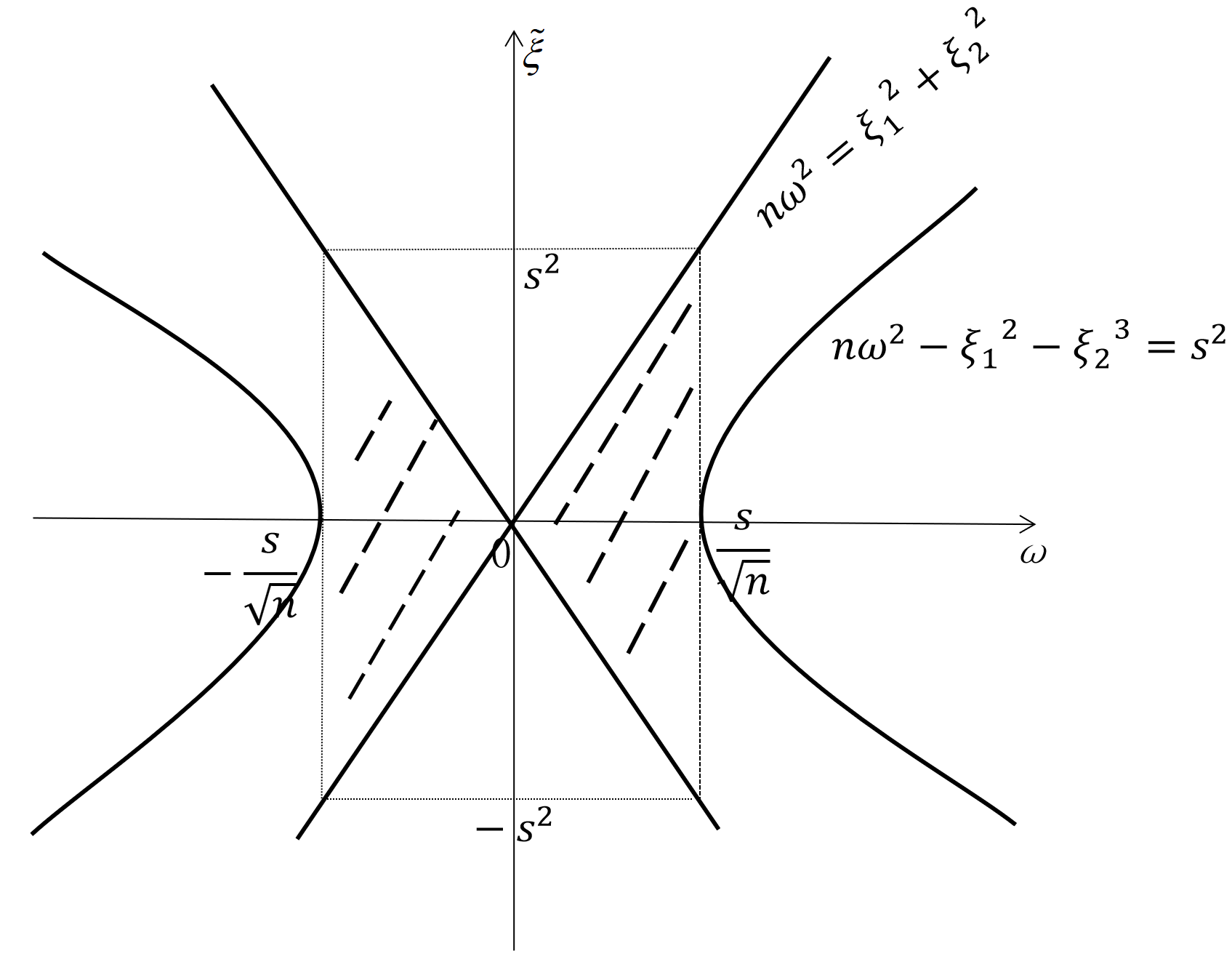}
	\caption{$E_s$ is the shadow area,\ $\tilde{\xi}^2=\xi_1^2+\xi_2^2$.}
	\label{Fig.E}
\end{figure}

Let $\xi_p=\xi=\kappa_p\textbf{d}=(\kappa_pd_1,\kappa_pd_2,\kappa_pd_3)$.
Combining (\ref{pfg}) and  $|\widehat{g}(\xi_3)|\geq\delta>0$ for $\xi_3\in(-b,b)$, we get that for $(\kappa_pd_1,\kappa_pd_2,\omega)\in E_b$ 
\begin{equation}\label{3hf}
\begin{split}
 |\textbf{p}\cdot\widehat{\textbf{f}}(\kappa_pd_1,\kappa_pd_2,\omega)|\leq C\frac{(1+\kappa_p)\epsilon}{|\hat{g}(\xi_3)|}&\leq C \delta^{-1} (1+\kappa_p) \epsilon\\
 &\leq\,C\,b\epsilon.
\end{split}
\end{equation}
Given
$$I(s)=\int_{E_s}|\textbf{p}\cdot\widehat{\textbf{f}}(\xi_1,\xi_2,\omega)|^2d\xi_1d\xi_2d\omega,$$ 
it is easy to verify that 
\begin{equation}
I(s)\leq\int_{E_s}|\widehat{\textbf{f}}(\xi_1,\xi_2,\omega)|^2d\xi_1d\xi_2d\omega.
\end{equation}
Using the polar coordinates $\xi_1=r\sin\theta$,  $\xi_2=r\cos\theta$ for
$0\leq r\leq s$,  $0\leq\theta\leq2\pi$, we obtain that
\begin{equation}
\nonumber
\begin{split}
\int_{E_s\cap\{(\xi_1,\xi_2,\omega)|\ \omega\geq0\}}|\widehat{\textbf{f}}|^2d\xi_1d\xi_2d\omega=\int_0^s\int_0^{2\pi}\int_{\frac{1}{\sqrt{n}}r}^{\frac{1}{\sqrt{n}}s}|\widehat{\textbf{f}}|^2r\,d\omega \,dr\,d\theta.
\end{split}
\end{equation}
Let $r=s\hat{r}$, $r\hat{\xi}=(\xi_1,\xi_2)$ and $\hat{\omega}\in(0,1)$. A simple calculation yields 
\begin{equation}
\nonumber
\begin{split}
&\int_{E_s\cap\{(\xi_1,\xi_2,\omega)|\ \omega\geq0\}}|\widehat{\textbf{f}}(\xi_1,\xi_2,\omega)|^2d\xi_1d\xi_2d\omega\\
&=\int_0^s\int_0^{2\pi}\int_{\frac{1}{\sqrt{n}}r}^{\frac{1}{\sqrt{n}}s}|\widehat{\textbf{f}}(\xi_1,\xi_2,\omega)|^2d\omega rdrd\theta\\
&=\int_0^1\int_0^{2\pi}\big(\int_{\frac{1}{\sqrt{n}}s\hat{r}}^{\frac{1}{\sqrt{n}}s}|\widehat{\textbf{f}}(r\hat{\xi},\omega)|^2d\omega\big)s^2\hat{r}d\hat{r}d\theta\\
&=\int_0^1\int_0^{2\pi}\big(\int_0^{\frac{1}{\sqrt{n}}s}|\widehat{\textbf{f}}(s\hat{r}\hat{\xi},\omega)|^2d\omega-\int_0^{\frac{1}{\sqrt{n}}s\hat{r}}|\widehat{\textbf{f}}(s\hat{r}\hat{\xi},\omega)|^2d\omega\big) s^2\hat{r}d\hat{r}d\theta\\
&=\int_0^1\int_0^{2\pi}\big(\int_0^1|\textbf{f}_1(s\hat{r}\hat{\xi},\frac{1}{\sqrt{n}}s\hat{\omega})|^2\frac{1}{\sqrt{n}}s^3\hat{r}d\hat{\omega}-\int_0^1|\textbf{f}_2(s\hat{r}\hat{\xi},\frac{1}{\sqrt{n}}s\hat{r}\hat{\omega})|^2\frac{1}{\sqrt{n}}s^3\hat{r}^2d\hat{\omega}\big) d\hat{r}d\theta
\end{split}
\end{equation}
where 
\begin{equation}
\nonumber
\textbf{f}_1(s\hat{r}\hat{\xi},\frac{1}{\sqrt{n}}s\hat{\omega})=\int_{\R^3}\textbf{f}(x_1,x_2,t)e^{-i(s\hat{r}\hat{\xi}\cdot\tilde{x}+\frac{1}{\sqrt{n}}s\hat{\omega} t)}d\tilde{x}dt
\end{equation}
and
\begin{equation}
\nonumber
\textbf{f}_2(s\hat{r}\hat{\xi},\frac{1}{\sqrt{n}}s\hat{r}\hat{\omega})=\int_{\R^3}\textbf{p}\cdot\textbf{f}(x_1,x_2,t)e^{-i(s\hat{r}\hat{\xi}\cdot\tilde{x}+\frac{1}{\sqrt{n}}s\hat{r}\hat{\omega}t)}d\tilde{x}dt.
\end{equation}
This integral is an analytic function of $s=s_1+is_2\in \C$, $s_1,s_2\in\R$.
Noting that 
$|-i(s\hat{r}\hat{\xi}\cdot\tilde{x}+\frac{1}{\sqrt{n}}s\hat{\omega} t|\leq (R+\frac{1}{\sqrt{n}}T)|s_2|$ and $|-i(s\hat{r}\hat{\xi}\cdot \tilde{x}+\frac{1}{\sqrt{n}}s\hat{r}\hat{\omega}t)|\leq (R+\frac{1}{\sqrt{n}}T)|s_2|$, we deduce 
\begin{equation}
\nonumber
\begin{split}
|I(s)|\leq CM^2|s|^3e^{((R+cT)+1)|s_2|}, \ s\in S
\end{split}
\end{equation}
where $c:=\frac{1}{\sqrt{n}}$ and $C>0$ depends on $n$, $T_0$ and $R$.

From estimate (\ref{3hf}), we have
\begin{equation}\label{3Cb}
\begin{split}
|I(s)|\leq C\,|E_s|\,b^2\,\epsilon^2 \quad \mbox{for all} \quad s\in[0,b],
\end{split}
\end{equation}
where $C>0$ depends on $n$, $\delta$, $T_0$ and $R$.
Then applying Lemma \ref{2}, we know for all $s>b$ that 
\begin{equation}
\nonumber
\begin{split}
|I(s)|\leq CM^2\,b^2\,e^{((R+cT)+1)s}\epsilon^{2\mu(s)},
\end{split}
\end{equation}
where $C>0$ depends on $n$, $\delta$, $T_0$ and $R$.

Define $E_s^1=\{(\xi_1,\xi_2,\omega)\in\R^3| \, |\omega|>\frac{1}{\sqrt{n}}s, \ \xi_1^2+\xi_2^2\leq n\omega^2 \}$. 
The Parseval's identity yields
\[I_1(s)=\int_{E^1_s}|\textbf{p}\cdot\widehat{\textbf{f}}(\xi_1,\xi_2,\omega)|^2d\xi_1d\xi_2d\omega\leq \frac{C}{s^2}\int_{E^1_s}|\widehat{\nabla (\textbf{p}\cdot\textbf{f})}(\xi_1,\xi_2,\omega)|^2d\xi_1d\xi_2d\omega \leq\,C\,\frac{M^2}{s^2},\]
where $C>0$ depends on $n$.
Consider the set 
\begin{equation}
\nonumber
E^2_s=\{(\xi_1,\xi_2,\omega)\in\R^3| \, \omega^2+\xi_1^2+\xi_2^2\leq \frac{1}{n}s^2, \ \xi_1^2+\xi_2^2\geq n\omega^2 \}.
\end{equation}
Now we estimate  
$$I_2(s)=\int_{E^2_s}|\textbf{p}\cdot\widehat{\textbf{f}}(\xi_1,\xi_2,\omega)|^2d\xi_1d\xi_2d\omega.$$
The following Lemma \ref{lem} is essential to estimate $I_2(s)$.
For $r>0$, denote $B(0,r)=\{x\in\R^d: |x|<r\}$. Below we state a stability estimate for analytic continuation problems, which can be seen in \cite{Mourad-Ben, Vessella}.
\begin{lemma}\label{lem}
Let $\mathcal{O}$ be a non empty open set of the unit ball $B(0,1)\subset
\mathbb{R}^{d}$, $d\geq2$, and let $G$ be an analytic function in $B(0,2),$ that satisfy
$$\|\partial^{\gamma}G\|_{L^{\infty}(B(0,2))}\leq M_0\,|\gamma|!\;{\eta^{-|\gamma|}},\,\,\,\,\forall\,\gamma\in(\mathbb{N}\cup\{0\})^{d},$$
for some  $M_0>0$ and $\eta>0$. Then, we have
$$\|G\|_{L^{\infty}(B(0,1))}\leq N\, M_0^{1-\mu}\;\|G\|_{L^{\infty}(\mathcal{O})}^{\mu},$$
where $\mu\in(0,1)$ depends on $d$, $\eta$ and $|\mathcal{O}|$ and $N=N(\eta)>0$.
\end{lemma}

Define
$F_s(\xi_1,\xi_2,\omega)=\textbf{p}\cdot\widehat{\textbf{f}}(s\xi_1,s\xi_2,s\omega)$ for any $(\xi_1,\xi_2,\omega)\in\mathbb{R}^3$ and $s>0$. Since $\textbf{f}$ is compactly supported, one can see that the function $F_s$ is analytic and it satisfies for $\gamma\in (\N\cup\{0\})^3$ that
\begin{eqnarray*}
|\partial^\gamma F_s(\xi_1,\xi_2,\omega)|=|\partial^\gamma F(s\xi_1,s\xi_2,s\omega)|=\Big|\partial^\gamma \int_{\mathbb{R}^3}\textbf{p}\cdot\textbf{f}(x_1,x_2,t) e^{-is(\xi_1\cdot x_1+\xi_2\cdot x_2+\omega t)}\,dx_1 dx_2dt\Big|\cr
= \Big|  \sum_{\alpha+\beta=\gamma \atop \alpha,\beta\in\N\cup0}\int_{\mathbb{R}^3} (-i)^{|\gamma|} s^{|\gamma|} \tilde{x}^{\alpha}t^{\beta} \textbf{p}\cdot\textbf{f}(x_1,x_2,t) e^{-ib(\xi_1\cdot x_1+\xi_2\cdot x_2+\omega t)}\,dx_1 dx_2dt \Big|.
\end{eqnarray*}
Using $s^{|\gamma|}<|\gamma|!\, e^s $ and  $\eta=\max\{R,T\}^{-1}$, 
we obtain
\begin{eqnarray}
\nonumber
|\partial^\gamma F_s(\xi_1,\xi_2,\omega)|\leq \|\textbf{f}\|_{L^2(B_R\times(0,T))}\, \eta^{-|\gamma|} \,s^{|\gamma|} 
\leq  C\|\textbf{f}\|_{L^2(B_R\times(0,T))}\,\eta^{-|\gamma|}\, |\gamma|!\, e^s \leq  C M\,\eta^{-|\gamma|}\, |\gamma|!\, e^s,
\end{eqnarray}
where $C>0$ depends on $R$ and $T$.
Applying Lemma \ref{lem} to the set $\mathcal{O}$ defined as $\mathcal{O}:= E_s^2$,
we find a constant $\alpha\in (0,1)$ such that
 $$\|F_s\|_{L^\infty(B(0,1))}\leq C\,M\, e^{s(1-\alpha)} \|F_s\|^\alpha_{L^\infty(\mathcal{O})}.$$
Using the fact that $\textbf{p}\cdot\widehat{\textbf{f}}(\xi)=F_s(s^{-1}\xi)$, one gets the following estimate
\begin{eqnarray}\nonumber
\|\textbf{p}\cdot\widehat{\textbf{f}}\|_{L^{\infty}(B(0,s))}&=&\|F_s\|_{L^\infty(B(0,1))}\leq CMe^{s(1-\alpha)} \|F_s\|_{L^\infty(\mathcal{O})}^\alpha\\
&\leq& C\,M\, e^{s(1-\alpha)} \|\textbf{p}\cdot\widehat{\textbf{f}}\|_{L^\infty(E_s^2)}^\alpha
\leq C\,M\,e^{s(1-\alpha)} b^{\alpha}\epsilon^\alpha,
\end{eqnarray}
where $C>0$ depends on $R$, $T$, $\delta$ and $n$.
Combining (\ref{3hf}), we know for $s\in[0,b]$ that 
\begin{equation}\nonumber
|I_2(s)|\leq CM^2|E^2_s|\,e^{2s(1-\alpha)}\, b^{2\alpha}\,\epsilon^{2\alpha}
\end{equation}
and similarly 
\begin{equation}\nonumber
\begin{split}
|I_2(s)|\leq CM^2|s|^3\,e^{\big((R+cT)+1\big)|s_2|} \quad \mbox{for all} \quad s\in S.
\end{split}
\end{equation}
Thus applying Lemma \ref{2}, we have 
\begin{equation}
\begin{split}\label{rct}
|e^{-((R+cT)+3)s}I_2(s)|\leq CM^2b^{2\alpha}(\epsilon^{\alpha})^{2\mu} \quad \mbox{for all} \quad s>b,
\end{split}
\end{equation}
where $C>0$ depends on $R$, $T$, $\delta$ and $n$.

Consider the set 
\begin{equation}\nonumber
E^3_s=\{(\xi_1,\xi_2,\omega)\in\R^3| \, \ \omega^2+|\xi_1|^2+|\xi_2|^2>\frac{1}{n}s^2, \ \xi_1^2+\xi_2^2\geq n\omega^2 \}.
\end{equation}
The Parseval's identity implies
\begin{equation}
I_3(s)=\int_{E^3_s}|\textbf{p}\cdot\widehat{\textbf{f}}(\xi_1,\xi_2,\omega)|^2d\xi_1d\xi_2d\omega\leq\frac{CM^2}{s^2},
\end{equation}
where $C>0$ depends on $n$.

Now we proof Theorem \ref{TH4}.
We assume that $\epsilon<e^{-1}$, otherwise the estimate is obvious. Let 
\begin{eqnarray}
s=
\begin{cases}
\frac{1}{(((R+cT)+3)\pi)^{\frac{1}{3}}}b^{\frac{2}{3}}|\ln\epsilon|^{\frac{1}{4}} & \textrm{if}\ |\ln\epsilon|^{\frac{1}{4}}>2^{\frac{1}{4}}b^{\frac{1}{3}}(((R+cT)+3)\pi)^{\frac{1}{3}},\\
b & \textrm{if} \ |\ln\epsilon|^{\frac{1}{4}}>2^{\frac{1}{4}}b^{\frac{1}{3}}(((R+cT)+3)\pi)^{\frac{1}{3}}.
\end{cases}
\end{eqnarray}

Case (i): $|\ln\epsilon|^{\frac{1}{4}}>2^{\frac{1}{4}}b^{\frac{1}{3}}(((R+cT)+3)\pi)^{\frac{1}{3}}$. 
One can check that 
\[s>2^{\frac{1}{4}}b.\]
Thus, using Lemma \ref{2} we obtain
\begin{eqnarray}
|I(s)|&\leq & CM^2\,b^2 e^{((R+cT)+3)s^2}\epsilon^{2\mu(s)}\cr
&\leq & CM^2\,b^2\,e^{-2\mu(s)|\ln\epsilon|+((R+cT)+3)s^2}\cr
&\leq & CM^2\,b^2\,e^{-\frac{-2|\ln\epsilon|}{\pi}(\frac{b}{s})^2+\frac{((R+cT)+3)}{(((R+cT)+3)\pi)^{\frac{1}{3}}}b^{\frac{2}{3}}|\ln\epsilon|^{\frac{1}{4}}}\cr
&\leq & CM^2\,b^2\, e^{-2\big(\frac{((R+cT)+3)^2}{\pi}\big)^{\frac{1}{3}}b^{\frac{2}{3}}|\ln\epsilon|^{\frac{1}{2}}\big(1-\frac{1}{2} |\ln\epsilon|^{-\frac{1}{4}} \big)},
\end{eqnarray}
where $C>0$ depends on $R$, $T$, $\delta$ and $n$.
Noting that 
$\frac{1}{2}|\ln\epsilon|^{-\frac{1}{4}}<\frac{1}{2}$ and $\big(\frac{(2(R+2T)+3)^2}{\pi}\big)^{\frac{1}{3}}>1$,
we have 
\[|I(s)|\leq C M^2\,b^2\,e^{-b^{\frac{2}{3}}|\ln\epsilon|^{\frac{1}{2}}}.\]
Using the elementary inequality 
\[e^{-t}\leq \frac{6!}{t^6}, \quad t>0,\]
we get 
\begin{equation}
|I(s)|\leq\frac{C\,b^2\,M^2}{\big(b^{\frac{2}{3}}|\ln\epsilon|^{\frac{1}{2}}\big)^6}\leq\frac{CM^2}{b^{2}|\ln\epsilon|^{3}},
\end{equation}
where $C>0$ depends on $R$, $T$, $\delta$ and $n$.

Case (ii): $|\ln\epsilon|^{\frac{1}{4}}>2^{\frac{1}{4}}b^{\frac{1}{3}}(((R+cT)+3)\pi)^{\frac{1}{3}}$. In this case we have $s=b $, and from (\ref{3Cb}),
\[|I(s)|=|I(b)| \leq |E_{b}|b^2\epsilon^2.\]
Combining the estimates of $I(s)$ and $I_1(s)$, we obtain 
\begin{eqnarray}
\int_{E_s}|\textbf{p}\cdot\widehat{\textbf{f}}(\tilde{\xi},\omega)|^2\,d\tilde{\xi}\,d\omega &+&\int_{E_s^1}|\textbf{p}\cdot\widehat{\textbf{f}}(\tilde{\xi},\omega)|^2\,d\tilde{\xi}\,d\omega
=I(s)+I_1(s)\cr
&\leq& C\Big(b^5\,b^2\epsilon^2+
+\frac{M^2}{b^{2}|\ln\epsilon|^{3}}+\frac{M^2}{\big(b^{\frac{2}{3}}|\ln\epsilon|^{\frac{1}{4}}\big)^2}\Big),
\end{eqnarray}
where $C>0$ depends on $R$, $T$, $\delta$ and $n$.

Now we estimate $I_2(s)+I_3(s)=\int_{E_s^2}|\textbf{p}\cdot\widehat{\textbf{f}}(\tilde{\xi},\omega)|^2\,d\tilde{\xi}\,d\omega+\int_{E_s^3}|\textbf{p}\cdot\widehat{\textbf{f}}(\tilde{\xi},\omega)|^2\,d\tilde{\xi}\,d\omega$.
Let 
\begin{eqnarray}\nonumber
s=
\begin{cases}
\frac{1}{(((R+cT)+3)\pi)^{\frac{1}{3}}}b^{\frac{2}{3}}|\alpha\ln\epsilon|^{\frac{1}{4}} & \textrm{if}\ |\alpha\ln\epsilon|^{\frac{1}{4}}>2^{\frac{1}{4}}b^{\frac{1}{3}}(((R+cT)+3)\pi)^{\frac{1}{3}},\\
b & \textrm{if} \ |\alpha\ln\epsilon|^{\frac{1}{4}}>2^{\frac{1}{4}}b^{\frac{1}{3}}(((R+cT)+3)\pi)^{\frac{1}{3}}.
\end{cases}
\end{eqnarray}
Repeating similar steps, we find out
\begin{eqnarray}
\nonumber
\int_{E_s^2}|\textbf{p}\cdot\widehat{\textbf{f}}(\tilde{\xi},\omega)|^2\,d\tilde{\xi}\,d\omega +\int_{E_s^3}|\textbf{p}\cdot\widehat{\textbf{f}}(\tilde{\xi},\omega)|^2\,d\tilde{\xi}\,d\omega=I_2(s)+I_3(s)\cr
\leq C\Big(b^5\,e^{2b(1-\alpha)}\epsilon^{2\alpha}
+\frac{M^2}{b^{2}|\alpha\ln\epsilon|^{3}}+\frac{M^2}{\big(b^{\frac{2}{3}}|\alpha\ln\epsilon|^{\frac{1}{4}}\big)^2}\Big),
\end{eqnarray}
where $C>0$ depends on $R$, $T$, $\delta$, $M$ and $n$.

Since $b^{2}|\ln\epsilon|^{3}>b^{\frac{4}{3}}|\ln\epsilon|^{\frac{1}{2}}$ and $b^{2}|\alpha\ln\epsilon|^{3}>b^{\frac{4}{3}}|\alpha\ln\epsilon|^{\frac{1}{2}}$ when $b>1$ and $|\ln\epsilon|>1$.
Applying $b^{\frac{4}{3}}|\ln\epsilon|^{\frac{1}{2}}>b^{\frac{4}{3}}|\alpha\ln\epsilon|^{\frac{1}{2}}$, we obtain 
\begin{equation}\label{pf1}
\begin{split}
\int_{\R^3}|\textbf{p}\cdot\widehat{\textbf{f}}|^2d\tilde{\xi} d\omega
&=\int_{E_s}|\textbf{p}\cdot\widehat{\textbf{f}}|^2\,d\tilde{\xi}\,d\omega+\int_{E_s^1}|\textbf{p}\cdot\widehat{\textbf{f}}|^2\,d\tilde{\xi}\,d\omega+\int_{E_s^2}|\textbf{p}\cdot\widehat{\textbf{f}}|^2\,d\tilde{\xi}\,d\omega+\int_{E_s^3}|\textbf{p}\cdot\widehat{\textbf{f}}|^2\,d\tilde{\xi}\,d\omega\\
&\leq C\Big(b^7\epsilon^2+b^5e^{2b(1-\alpha)}\,\epsilon^{2\alpha}+\frac{M^2}{b^2|\ln\epsilon|^3}+\frac{M^2}{b^{\frac{4}{3}}|\ln\epsilon|^{\frac{1}{2}}}+\frac{M^2}{b^2|\alpha\ln\epsilon|^3}+\frac{M^2}{b^{\frac{4}{3}}|\alpha\ln\epsilon|^{\frac{1}{2}}}\Big)\\
&\leq C(b^7\epsilon^2+b^5e^{2b(1-\alpha)}\,\epsilon^{2\alpha}
+\frac{M^2}{b^{\frac{4}{3}}|\alpha\ln\epsilon|^{\frac{1}{2}}}),
\end{split}
\end{equation}
where $C>0$ depends on $R$, $T$, $\delta$, $M$ and $n$.

Repeating similar steps, we get that 
\begin{equation}\label{pf2}
\int_{\R^3}|\textbf{q}\cdot\widehat{\textbf{f}}|^2d\tilde{\xi} d\omega\leq  C(b^7\epsilon^2+b^5e^{2b(1-\alpha)}\,\epsilon^{2\alpha}
+\frac{M^2}{b^{\frac{4}{3}}|\alpha\ln\epsilon|^{\frac{1}{2}}}),
\end{equation}
where $C>0$ depends on $R$, $T$, $\delta$, $M$ and $n$.
Using the Pythagorean theorem yields 
\begin{equation}\label{pqf}
\begin{split}
||\textbf{f}||^2_{L^2(\R^3)^3}&=||\widehat{\textbf{f}}||^2_{L^2(\R^3)^3}\\
&=\int_{\R^3}|\textbf{p}\cdot\widehat{\textbf{f}}|^2d\tilde{\xi} d\omega+\int_{\R^3}|\textbf{q}\cdot\widehat{\textbf{f}}|^2d\tilde{\xi} d\omega.
\end{split}
\end{equation}
Combining (\ref{pf1})-(\ref{pf2}) and (\ref{pqf}), we obtain the stability estimate (\ref{es3}).
This completes the proof.


\section{Acknowledgment}
The work of S. Si is supported by the Natural Science Foundation of Shandong Province, China(No. ZR202111240173). The author would like to thank Guanghui Hu for helpful discussions.

%

\end{document}